\documentclass[11pt,twoside,a4paper]{article}

\author{Shai Sarussi}
 \title{Maximal covers of chains of prime ideals}
\date{}

\usepackage{amsmath,amsthm,amssymb}
 \usepackage{verbatim}

\begin{document}

\newtheorem{thm}{Theorem}[section]
\newtheorem{cor}[thm]{Corollary}
\newtheorem{lem}[thm]{Lemma}
\newtheorem{prop}[thm]{Proposition}
\newtheorem{ax}{Axiom}

\theoremstyle{definition}
\newtheorem{defn}[thm]{Definition}

\theoremstyle{remark}
\newtheorem{rem}[thm]{Remark}
\newtheorem{ex}[thm]{Example}
\newtheorem*{notation}{Notation}

\newcommand{\qv}{{quasi-valuation\ }}


\maketitle

\begin{abstract} { Suppose $f:S \rightarrow R$ is a ring homomorphism such that $f[S] $ is contained in the center of $R$.
We study the connections between chains in $\text{Spec} (S)$ and $\text{Spec} (R)$. We focus on the properties
LO (lying over), INC (incomparability), GD (going down), GU (going up) and SGB (strong going between).
We provide a sufficient condition for every maximal chain in $\text{Spec} (R)$ to cover a maximal chain in $\text{Spec} (S)$.
We prove some necessary and sufficient conditions for $f$ to satisfy each of the properties GD, GU and SGB, in terms of maximal $\mathcal D$-chains, where $\mathcal D \subseteq \text{Spec} (S)$ is a nonempty chain.
We show that if $f$ satisfies all of the above properties, then every maximal $\mathcal D$-chain is a perfect maximal cover of $\mathcal D$.
One of the main results of this paper is Corollary \ref{equivalent conditions}, in which we give equivalent conditions for the following property:
for every chain $\mathcal D \subseteq \text {Spec} (S)$ and for every maximal $\mathcal
D$-chain $\mathcal C \subseteq \text {Spec} (R)$, $\mathcal C$ and $\mathcal D$ are of the same cardinality.

}\end{abstract} 

\section*{Introduction}

All rings considered in this paper are nontrivial rings with identity; homomorphisms need not be unitary, unless otherwise stated.
All chains considered are chains with respect to containment.
The symbol $\subset$ means proper inclusion and the symbol $\subseteq$ means inclusion
or equality. For a ring $R$ we denote the center of $R$ by $Z(R)$.

We start with some background on the notions we consider.

In his highly respected paper from 1937, in which Krull (see [Kr]) proved his basic theorems regarding the behavior of prime ideals under integral extensions, and defined the notions LO, GU, GD and INC, he proposed the following question: assuming $S \subseteq R$ are integral domains with $R$ integral over $S$ and $S$ integrally closed; do adjacent prime ideals
in $R$ contract to adjacent prime ideals in $S$? In 1972 Kaplansky (see [Ka]) answered the question negatively. In 1977, Ratliff (see [Ra]) made another step; he defined and studied the notion GB (going between): let $S \subseteq R$ be commutative rings. $S \subseteq R$ is said to satisfy GB if whenever $Q_1 \subset Q_2$ are prime ideals of $R$ and there exists a prime ideal $Q_1 \cap S \subset P' \subset Q_2 \cap S$, then there exists a prime ideal $Q' $ in $R$ such that
$Q_1 \subset Q' \subset Q_2$. Later, in 2003, G. Picavet (see [Pi])
introduced the notion SGB (strong going between), to be presented later.


In 2003, Kang and Oh (cf. [KO]) defined the notion of an SCLO extension: let $S \subseteq R$ be commutative rings. $S \subseteq R$ is called an
SCLO extension if, for every chain of prime ideals $\mathcal D$ of $S$ with an initial element $P$ and $Q$ a prime ideal of $R$ lying over $P$, there exists a chain of prime ideals $\mathcal C$ of $R$ lying over $\mathcal D$ whose initial element is $Q$. They proved that $S \subseteq R$ is
an SCLO extension iff it is a GU extension (cf. [KO, Corollary 12]). In particular, if $S \subseteq R$ is a GU extension, then for every chain of prime ideals $\mathcal D$ of $S$ there exists a chain of prime ideals of $R$ covering $\mathcal D$. We note that when considering a unitary homomorphism $f:S \rightarrow R$, some authors call
$f$ a chain morphism if every chain in $\text{Spec}(S)$ can be covered by a chain in $\text{Spec}(R)$. Also, some authors call an SCLO extension a GGU (generalized going up) extension. It follows that, in the commutative case, GU implies GGU for unitary ring homomorphisms. In 2005 Dobbs and Hetzel (cf. [DH]) proved that, in the commutative case,
GD implies GGD (generalized going down) for unitary ring homomorphisms. GGD is defined analogously to GGU.

We consider a more general setting. Namely, {\it In this paper $S$ and $R$ are rings (not necessarily commutative) and $f:S \rightarrow R$ is a homomorphism
(not necessarily unitary) such that $f[S] \subseteq Z(R)$.} We study the connections between chains of prime ideals of $S$ and chains of prime ideals of $R$.

For $I \vartriangleleft R$ and $J \vartriangleleft S$ we say that $I$ is lying over $J$ if $J=f^{-1}[I]$.
Note that if $Q$ is a prime ideal of $R$ then $f^{-1}[Q]$ is a prime ideal of $S$ or $f^{-1}[Q]=S$.
It is not difficult to see that $f$ is unitary if and only if for all $Q \in \text{Spec} (R)$, $f^{-1}[Q] \in \text{Spec} (S)$.
Indeed, $(\Rightarrow)$ is easy to check and as for $(\Leftarrow)$, if $f$ is not unitary then $f(1)$ is a central idempotent$\neq 1$. Thus, there exists a prime ideal
$Q \in \text {Spec} (R)$ containing $f(1)$; clearly, $f^{-1}[Q]=S$.

For the reader's convenience, we now define the five basic properties we consider.

 We say that $f$ satisfies LO (lying over) if for all $P \in \text {Spec} (S)$ there exists $Q \in \text {Spec}(R)$ lying over $P$.


 We say that $f$
satisfies GD (going down) if for any $P_{1} \subset
P_{2}$ in $\text{Spec} (S)$ and for every $Q_{2} \in \text{Spec} (R)$
lying over $P_{2}$, there exists $Q_{1} \subset Q_{2}$ in
$\text{Spec} (R)$ lying over $P_{1}$.

We say that $f$
satisfies GU (going up) if for any $P_{1} \subset
P_{2}$ in $\text{Spec} (S)$ and for every $Q_{1} \in \text{Spec} (R)$
lying over $P_{1}$, there exists $Q_{1} \subset Q_{2}$ in
$\text{Spec} (R)$ lying over $P_{2}$.

We say that $f$
satisfies SGB (strong going between) if for any $P_{1} \subset
P_{2} \subset
P_{3}$ in $\text{Spec} (S)$ and for every $Q_{1} \subset
Q_{3}$ in $\text{Spec} (R)$ such that $Q_{1}$ is lying over $P_{1}$ and
$Q_{3}$ is lying over $P_{3}$, there exists $Q_{1} \subset Q_{2} \subset Q_{3}$ in
$\text{Spec} (R)$ lying over $P_{2}$.

We note that the "standard" definition of INC is as follows: $f$ is said to satisfy INC (incomparability) if
whenever $Q_{1} \subset Q_{2}$ in $\text{Spec}(R)$, we have $f^{-1}[Q_{1}]
 \subset f^{-1}[Q_{2}] $. However, in order for us to be able to discuss non-unitary homomorphisms we make the following modification and
 define INC as follows: we say that $f$ satisfies INC if
whenever $Q_{1} \subset Q_{2}$ in $\text{Spec}(R)$ and $f^{-1}[Q_{2}] \neq S$, we have $f^{-1}[Q_{1}]
 \subset f^{-1}[Q_{2}] $.

We note that the case in which $R$ is an algebra over $S$ can be considered as a special case (of the case we consider). In particular, the case in which $S \subseteq R$ are rings and $S \subseteq Z(R)$ (and thus $S$ is a commutative ring) can be considered as a special case. Also, in this special case, one can show that, as in the commutative case, GU implies LO. However, in our case, GU does not necessarily imply LO. As an easy example, let $p \in \Bbb N$ be a prime number and let $2 \leq n \in \Bbb N$ such that $n$ is not a power of $p$. Let $f: \Bbb Z_{np} \rightarrow \Bbb Z_{p}$ be the homomorphism defined by $x$mod$(np)$ $\rightarrow x$mod$(p)$ . Then $f$ satisfies GU (in a trivial way, because k-dim $\Bbb Z_{np}=0$) but there is no prime ideal in $\Bbb Z_p$ lying over $q\Bbb Z_{np}$, where $p \neq q \in \Bbb N$ is any prime number dividing $n$.

\section{Maximal chains in Spec$(R)$ that cover maximal chains in Spec$(S)$}

In this section we prove that every maximal chain of
prime ideals of $R$ covers a maximal chain of prime ideals of $S$,
under the assumptions GU, GD and SGB. We also present an example from quasi-valuation theory.

We start with some basic definitions.

\begin{defn} Let $\mathcal D$ denote a set
of prime ideals of $S$. We say that a chain $\mathcal
C$ of prime ideals of $R$ is a {\it
$\mathcal D$-chain} if $f^{-1}[Q] \in \mathcal D$ for
every $Q \in \mathcal
C$.\end{defn}

\begin{defn} Let $\mathcal D \subseteq \text {Spec}(S)$ be a chain of prime ideals
and let $\mathcal C  \subseteq \text {Spec}(R)$ be a $\mathcal D$-chain. We say that $\mathcal C$ is a {\it cover} of
$\mathcal D$ (or that $\mathcal C$ covers $\mathcal D$) if for all $P \in \mathcal D$ there exists
$Q \in \mathcal C$ lying over $P$; i.e. the map $Q \rightarrow f^{-1}[Q]$ from $\mathcal C$ to $\mathcal D$ is surjective.
We say that $\mathcal C$ is a {\it perfect cover} of
$\mathcal D$ if $\mathcal C$ is a cover of $\mathcal D$ and the map $Q \rightarrow f^{-1}[Q]$ from $\mathcal C$ to $\mathcal D$ is injective. In other words, $\mathcal C$ is a perfect cover of
$\mathcal D$ if the map $Q \rightarrow f^{-1}[Q]$ from $\mathcal C$ to $\mathcal D$ is bijective. \end{defn}

The following lemma is well known.

\begin{lem} \label{alphar in Q implies one is in Q} Let $R$ be a ring. Let $\alpha \in Z(R)$, $r \in R$ and $Q \in \text{Spec}(R)$. If
$\alpha r \in Q$ then $\alpha \in Q$ or $r \in Q$.
\end{lem}

\begin{proof} $\alpha R r= R \alpha r  \subseteq Q$. Thus $\alpha \in Q$ or $r \in Q$.

\end{proof}

We recall now the notion of a weak zero divisor, introduced in [BLM].
For a ring $R$, $a \in R$ is called a weak zero-divisor if there
are $r_1,r_2 \in R$ with $r_1ar_2 = 0$ and $r_1r_2 \neq 0$. In [BLM] it is shown that in any ring R,
the elements of a minimal prime ideal are weak zero-divisors. Explicitly, the following is proven (see [BLM, Theorem 2.2]).

\begin{lem} Let $R$ be a ring and let $Q \in \text{Spec}(R)$ denote a minimal prime of $R$.
Then for every $q \in Q$, $q$ is a weak zero divisor.

\end{lem}

It will be more convenient for us to use the following version of the previous lemma.

\begin{lem} \label{elements of a minimal prime ideal are weak zero-divisors}
Let $R$ be a ring, let $I \vartriangleleft R$ and let $Q \in \text{Spec}(R)$ denote a minimal prime ideal over $I$.
Then for every $q \in Q$ there exist $r_1,r_2 \in R$ such that $r_1qr_2 \in I$ and $r_1r_2 \notin I$.

\end{lem}

We return now to our discussion. But before proving the next proposition, we note that $R$ is not necessarily commutative.
Therefore, a union of prime ideals of a chain in $\text{Spec}(R)$ is not necessarily a prime ideal of $R$.

\begin{prop} \label{minimal prime over union of chain intersection is the same} Let $\mathcal C \subseteq \text{Spec}(R)$ denote a chain of prime ideals of $R$. Let $\bigcup_{Q \in \mathcal C} Q \subseteq Q' \in \text{Spec}(R)$
denote a minimal prime ideal over $\bigcup_{Q \in \mathcal C} Q $. Then $f^{-1}[\bigcup_{Q \in \mathcal C} Q]=f^{-1}[Q']$.
In particular, if $f$ is unitary then $f^{-1}[\bigcup_{Q \in \mathcal C} Q]$ is a prime ideal of $S$.

\end{prop}

\begin{proof} Assume to the contrary that $f^{-1}[\bigcup_{Q \in \mathcal C} Q] \subset f^{-1}[Q']$ and let
$s \in f^{-1}[Q'] \setminus f^{-1}[\bigcup_{Q \in \mathcal C} Q]$. By Lemma \ref{elements of a minimal prime ideal are weak zero-divisors}, there exist $r_1,r_2 \in R$ such that $r_1f(s)r_2 \in \bigcup_{Q \in \mathcal C} Q$
and $r_1r_2 \notin \bigcup_{Q \in \mathcal C} Q$. However, $f(s) \in f[S] \subseteq Z(R)$ and thus $f(s)r_1r_2=r_1f(s)r_2 \in \bigcup_{Q \in \mathcal C} Q$.
So, there exists $Q \in \mathcal C$ such that $f(s)r_1r_2 \in Q$. Note that $f(s) \notin \bigcup_{Q \in \mathcal C} Q$; hence,
by Lemma \ref{alphar in Q implies one is in Q}, $r_1r_2 \in Q$. That is,
$r_1r_2 \in \bigcup_{Q \in \mathcal C} Q$, a contradiction.

\end{proof}

As shown, a union of prime ideals of a chain in $\text{Spec}(R)$ and a minimal prime ideal over it, are lying over the same prime ideal of $S$ (or over $S$).
Now, given an ideal $I$ of $R$ that is contained in an arbitrary prime ideal $Q_2$ of $R$,
we prove the existence of a minimal prime ideal over $I$, contained in $Q_2$.

\begin {lem} \label{there exists a minimal prime} Let $I$ be an ideal of $R$ and let $Q_{2}$ be any prime ideal of
$R$ containing $I$. Then there exists $I \subseteq Q' \subseteq Q_{2}$,
a minimal prime ideal over $I$. In particular, if $J$ is any ideal of $R$ then there exists a minimal prime ideal over $J$.
\end {lem}

\begin {proof} By Zorn's Lemma there exists a maximal chain
of prime ideals, say $\mathcal C'$, between $I$ and $Q_{2}$. So, $I \subseteq  \bigcap_{Q \in \mathcal C'} Q \subseteq Q_{2}$ is a minimal prime ideal over $I$. The last assertion is clear.

\end {proof}


We have proven the following: let $\mathcal C \subseteq \text{Spec}(R)$ denote a chain of prime ideals of $R$
and let $Q_2$ be a prime ideal of $R$ containing $\bigcup_{Q \in \mathcal C} Q$. Then there exists a prime ideal
$\bigcup_{Q \in \mathcal C} Q \subseteq Q' \subseteq Q_{2}$ such that $Q'$ is lying over $f^{-1}[\bigcup_{Q \in \mathcal C} Q]$, which is a prime ideal of $S$ or equals $S$.
One may ask whether this fact can be generalized as follows:
let $I$ be an ideal of $R$ lying over a prime ideal of $S$ and let $I \subseteq Q_{2}$ be any prime ideal of $R$.
Must there be a prime ideal $I \subseteq Q' \subseteq Q_{2}$ in $R$ such that $Q'$ is lying over the prime ideal $f^{-1}[I]$?
The answer to this question is: "no", even in the case where $S \subseteq R$ are commutative rings and $R$ is integral over $S$. As shown in the following example.

\begin{ex} \label{ex GU LO but no GD} Let $S \subseteq R$ be rings such that $R$ satisfies LO over $S$ but not GD (in our terminology, assume that the map $f: S \rightarrow R$, defined by $f(s)=s$ for all $s \in S$,  satisfies LO but not GD). Let $P_{1} \subset P_{2}$ be prime ideals of $S$ and let $Q_{2}$ be a prime ideal of $R$ lying over $P_{2}$ such that
there is no prime ideal $Q_{1} \subset Q_{2}$ lying over $P_{1}$. Let $Q_{3}$ be a prime ideal of $R$ lying over $P_{1}$. Denote $I= Q_{3} \cap Q_{2}$ and note that $I$ is clearly not a prime ideal of $R$; then, $I$ is an ideal of $R$ lying over the prime ideal $P_{1}$ and $I \subset Q_{2}$, but there is no prime ideal $I \subseteq Q' \subset Q_{2}$ of $R$ lying over $P_{1}$.
\end{ex}

The following two remarks are obvious.

\begin{rem} \label{order preserving} Let $\mathcal C$ denote a chain of prime ideals of $R$. Then the map
$Q \rightarrow f^{-1}[Q]$ from $\mathcal C$ to $\text {Spec}(S)$ is order preserving.

\end{rem}


\begin{rem} \label{mini inc} Assume that $f$ satisfies INC. Let $\mathcal C$ denote a chain of prime ideals of $R$ and let
$\mathcal C'=\{ Q \in \mathcal C \mid  f^{-1}[Q] \neq S\}$. Then the map
$Q \rightarrow f^{-1}[Q]$ from $\mathcal C'$ to $\text{Spec}(S)$ is injective.

\end{rem}



For the reader's convenience, in order to establish the following results in a clearer way, we briefly present here the notion of cuts.

Let $T$ denote a totally ordered set. A subset $X$ of $T$ is called {\it initial} (resp.
{\it final}) if for every $\gamma \in X$ and $\alpha \in T$, if $\alpha
\leq \gamma$ (resp. $\alpha \geq \gamma$), then $\alpha \in
X$. A cut $\mathcal A=(\mathcal A^{L}, \mathcal A^{R})$ of
$T$ is a partition of $T$ into two subsets $\mathcal A^{L}$ and $\mathcal
A^{R}$, such that, for every $\alpha \in \mathcal A^{L}$ and $\beta
\in \mathcal A^{R}$, $\alpha<\beta$. To define a cut, one often writes $\mathcal A^{L}=X$, meaning that
$\mathcal A$ is defined as $(X, T \setminus X)$ when $X$ is an initial
subset of $T$. For more information about cuts see, for example, [FKK] or [Weh].

\begin{lem} \label{QsbsetQ} Let $\mathcal C=\{X_{\alpha}\}_{\alpha \in I}$ be a chain of subsets
of $R$. Let $\mathcal D=\{f^{-1}[X_{\alpha } ] \}_{\alpha \in I}$
and let $Y \subseteq S$, $Y \notin \mathcal D$. Let $A=\{X \in
\mathcal C \mid f^{-1}[X] \subset Y\}$ and $B=\{X \in \mathcal C
\mid f^{-1}[X] \supset Y\}$. If $X \in A$ and $X' \in \mathcal C
\setminus A$ then $X \subset X'$. If $X  \in B$ and $X' \in
\mathcal C \setminus B$ then $X \supset X'$. In particular,
assuming $A,B \neq \emptyset$ and denoting $X_{1}=\bigcup_{X \in
A} X$ and $X_{2}=\bigcap_{X \in B} X$, one has $X_{1} \subseteq
X_{2}$.

\end{lem}

\begin{proof} Clearly, $A$ is an initial subset of $\mathcal C$ and $B$ is a final subset of $\mathcal C$.
The assertions are now obvious.



\end{proof}

We shall freely use Lemma \ref{QsbsetQ} without reference.

We note that a maximal chain of prime ideals exists by Zorn's
Lemma and is nonempty, since every ring has a maximal ideal. Also
note that if $\mathcal C=\{Q_{\alpha}\}_{\alpha \in I}$ is a chain
of prime ideals of $R$ then $\mathcal D=\{f^{-1}[Q_{\alpha}]\}_{\alpha \in I} \subseteq \text {Spec}(S) \cup \{ S \}$ is a chain;
if $f$ is unitary then $\mathcal D$ is a chain of prime ideals of $S$. Our
immediate objective (to be reached in theorem \ref{D is a maximal chain}) is to prove that, under certain assumptions, if
$\mathcal C$ is a maximal chain in $\text {Spec}(R)$ then $\mathcal D$ is a
maximal chain in $\text {Spec}(S)$.

\begin{rem} \label{Q2 in C} Let $\mathcal C$ be a maximal
chain of prime ideals of $R$, let $A \neq \emptyset$ be an initial subset of $\mathcal C$ and let $B \neq \emptyset$ be a final subset of $\mathcal C$.
Then $\bigcap_{Q \in B} Q \in \mathcal C$. Moreover, if $\bigcup_{Q \in A} Q $ is a prime ideal of $R$ then $\bigcup_{Q \in A} Q \in \mathcal C$.

\end{rem}

\begin{proof} $\bigcap_{Q \in B} Q \in \text {Spec}(R)$ and $\mathcal C \cup \{ \bigcap_{Q \in B} Q \}$ is a chain; thus
$\bigcap_{Q \in B} Q \in \mathcal C$. In a similar way, if $\bigcup_{Q \in A} Q $ is prime then $\mathcal C \cup \{ \bigcup_{Q \in A} Q \}$ is a chain
of prime ideals in $\text {Spec}(R)$; thus
$\bigcup_{Q \in A} Q \in \mathcal C$.

\end{proof}

\begin{lem} \label{Q1 in A Q2 in B} Let $\mathcal C=\{Q_{\alpha}\}_{\alpha \in I}$ be a maximal chain of prime ideals of $R$,
$\mathcal D=\{f^{-1}[Q_{\alpha }]  \}_{\alpha \in I}$, $P \in
\text {Spec}(S) \setminus \mathcal D$, $A=\{Q \in \mathcal C \mid f^{-1}[Q]
\subset P\}$ and $B=\{Q \in \mathcal C \mid f^{-1}[Q] \supset P\}$.
Assume that $A,B \neq \emptyset$ and denote $Q_{1}=\bigcup_{Q \in A} Q$ and
$Q_{2}=\bigcap_{Q \in B} Q$. Then $Q_{2} \in B$ and $Q_{1} \in A$.

\end{lem}

\begin{proof}

We start by proving that $Q_{2} \in B$. First note that $f^{-1}[Q_{2}] \supseteq P$. By Remark \ref{Q2 in C}, $Q_{2} \in \mathcal C$.
Since $P \in \text {Spec}(S) \setminus \mathcal D$, one cannot have $f^{-1}[Q_{2}] = P$. Thus, $Q_{2} \in B$.

We shall prove now that $Q_{1}$ is a prime ideal of $R$. Assume to the contrary and let
$Q_{3}=\bigcap_{Q \in \mathcal C \setminus A} Q$. By Remark \ref{Q2 in C}, $Q_{3} \in \mathcal C$. Since $\mathcal C$ is a maximal chain
in $\text {Spec}(R)$, $Q_{3}$ is a minimal prime ideal over $Q_{1}$, strictly containing it.
Clearly $f^{-1}[Q_{1}] \subseteq P$ and by Proposition \ref{minimal prime over union of chain intersection is the same}, $f^{-1}[Q_{1}]=f^{-1}[Q_{3}]$. Since $P \in \text {Spec}(S) \setminus \mathcal D$, one cannot have $f^{-1}[Q_{3}] = P$.
Therefore, $f^{-1}[Q_{3}] \subset P$, i.e., $Q_{3} \in A$, a contradiction (to the fact that $Q_{3}$ strictly contains $Q_{1}$). So, $Q_{1}$ is a prime ideal of $R$. We conclude by Remark \ref{Q2 in C} that $Q_{1} \in \mathcal C$. Thus, $Q_{1} \in A$.

\end{proof}

Note that by the previous lemma, $Q_{1}$ is the greatest member of
$A$ and $Q_{2}$ is the smallest member of $B$.

\begin{thm} \label{D is a maximal chain} Assume that $f$ is unitary and satisfies GU, GD and SGB. Let $\mathcal C=\{Q_{\alpha}\}_{\alpha \in
I}$ be a maximal chain of prime ideals in $\text {Spec}(R)$. Then $\mathcal
D=\{f^{-1}[Q_{\alpha}]\}_{\alpha \in I}$ is a maximal chain of
prime ideals in $\text {Spec}(S)$. In other words, every maximal chain in $\text {Spec}(R)$ is a cover of some \textbf{maximal} chain in $\text {Spec}(S)$.

\end{thm}

\begin{proof} First note that since $f$ is unitary, every prime ideal of $R$ is lying over some prime ideal of $S$.
Now, since $\mathcal C$ is a maximal chain in $\text {Spec}(R)$, $\bigcup_{\alpha \in I}
Q_{\alpha}$ is a maximal ideal of $R$ containing each $Q_{\alpha}
\in \mathcal C$ and thus $\bigcup_{\alpha \in I} Q_{\alpha} \in
\mathcal C$. In a similar way, since $\bigcap_{\alpha \in I}
Q_{\alpha}$ is a prime ideal of $R$ contained in each $Q_{\alpha}
\in \mathcal C$, $\bigcap_{\alpha \in I} Q_{\alpha} \in \mathcal
C$. We prove that $\bigcap_{\alpha \in I} Q_{\alpha}$ is lying
over a minimal prime ideal of $S$. Indeed, $\bigcap_{\alpha \in I}
Q_{\alpha}$ is lying over a prime ideal $P$ of $S$; assume to the
contrary that there exists a prime ideal $P_{0} \subset P$. Then
by GD, there exists $Q_{0} \subset \bigcap_{\alpha \in I}
Q_{\alpha}$ lying over $P_{0}$. Thus, $\mathcal C \cup \{ Q_{0}\}$
is chain of prime ideals strictly containing $\mathcal C$, a contradiction.
Similarly, we prove that $\bigcup_{\alpha \in I} Q_{\alpha}$ is
lying over a maximal ideal of $S$. Indeed, $\bigcup_{\alpha \in I}
Q_{\alpha}$ is lying over a prime ideal $P$ of $S$; 
assume to the contrary that there exists a prime ideal $P \subset P'$. Then by
GU, there exists $\bigcup_{\alpha \in I} Q_{\alpha} \subset Q'  $
lying over $P'$. Thus, $\mathcal C \cup \{ Q'\}$ is a chain of prime ideals strictly
containing $\mathcal C$, a contradiction.

We prove now that $\mathcal D=\{f^{-1}[Q_{\alpha}]\}_{\alpha \in
I}$ is a maximal chain of prime ideals of $S$. Assume to the
contrary that there exists $P \in \text {Spec}(S)$ such that $\mathcal D
\cup \{ P\}$ is a chain of prime ideals strictly containing
$\mathcal D$. Denote $A=\{Q \in \mathcal C \mid f^{-1}[Q] \subset
P\}$ and $B=\{Q \in \mathcal C \mid f^{-1}[Q] \supset P\}$. Note
that by the previous paragraph $P$ cannot be the greatest element
nor the smallest element in $\mathcal D \cup \{ P\}$; therefore
$A,B \neq \emptyset$. Denote $Q_{1}=\bigcup_{Q \in A} Q$ and
$Q_{2}=\bigcap_{Q \in B} Q$ and let $P_{1}=f^{-1}[Q_{1}]$ and
$P_{2}=f^{-1}[Q_{2}]$; by Lemma \ref{Q1 in A Q2 in B}, $Q_{1} \in A$ and $Q_{2} \in B$.
Hence, $P_{1} \subset P \subset P_{2}$. Therefore, by SGB, there exists a prime ideal
$Q_{1} \subset Q \subset Q_{2}$ lying over $P$. It is easy to see
that $\mathcal C$ is a disjoint union of $A$ and $B$ because
$\mathcal D \cup \{ P \}$ is a chain. Now, by the definition of
$Q_{1}$ and $Q_{2}$, we get $Q' \subseteq Q_{1} \subset Q \subset
Q_{2} \subseteq Q''$ for every $Q' \in A$ and $Q'' \in \mathcal C
\setminus A=B$. Thus, $\mathcal C \cup \{ Q\}$ is a chain strictly
containing $\mathcal C$, a contradiction to the maximality of
$\mathcal C$.



\end{proof}

It is easy to see that if $f$ is not unitary then Theorem \ref{D is a maximal chain} is not valid. As a trivial example, take $f$
as the zero map. Obviously, $f$ satisfies GU, GD and SGB, in a trivial way. Here is a less trivial example: let $R$ be a ring and consider the homomorphism $f: R \rightarrow R \oplus R$ sending each $r \in R$ to $(r,0) \in R \oplus R$. It is easy to see that $f$ satisfies GU, GD and SGB. Now, let $\{Q_{\alpha}\}_{\alpha \in
I}$ be a maximal chain of prime ideals in $\text {Spec}(R)$; then $\mathcal C=\{R \oplus Q_{\alpha}\}_{\alpha \in
I}$ is a maximal chain of prime ideals in $\text {Spec}(R \oplus R)$. However $ \{ f^{-1}[R \oplus Q_{\alpha}]\}_{\alpha \in
I} =\{ R \}$, which is clearly not a maximal chain of prime ideals of $R$.

\begin{cor} \label{one to one} Assume that $f$ is unitary and satisfies INC, GU, GD and SGB. Let $\mathcal C=\{Q_{\alpha}\}_{\alpha \in
I}$ be a maximal chain of prime ideals in $\text {Spec}(R)$. Then $\mathcal C$ is a perfect cover of
the \textbf {maximal} chain $\mathcal D=\{f^{-1}[Q_{\alpha}]\}_{\alpha \in I}$.


\end{cor}

\begin{proof} By Lemma \ref{mini inc} and Theorem \ref{D is a maximal chain}.
\end{proof}

\begin{ex} \label{example from qv theory} Suppose $F$ is a field with valuation $v$ and valuation ring $O_{v}$, $E/F$ is a finite dimensional field extension
and $R \subseteq E$ is a subring of $E$ lying over $O_{v}$. By [Sa1, Theorem 9.34] there exists a quasi-valuation $w$ on $RF$ extending the valuation $v$, with $R$ as its quasi-valuation ring; in this case $R$ satisfies LO (although we do not need this property here), INC, and GD over $O_{v}$ (see [Sa1, Theorem 9.38, 3]). By [Sa3, Theorem 3.7]
$R$ satisfies SGB over $O_{v}$. Moreover,
if there exists a quasi-valuation extending $v$, having a value group, and such that $R$ is the quasi-valuation ring,
then $R$ satisfies GU over $O_{v}$ (see [Sa1, Theorem 9.38, 6a]). Now, let $\mathcal
C=\{Q_{\alpha} \}_{\alpha \in I}$ be \textbf {any maximal
chain} of prime ideals of $R$. By Corollary \ref{one to one}, the map $Q
\rightarrow Q \cap O_{v}$ is a bijective order preserving
correspondence between $\mathcal C$ and \textbf {the maximal
chain} $\mathcal D=\{Q_{\alpha} \cap S\}_{\alpha \in I}$ of prime
ideals of $O_{v}$, namely, $\mathcal D=\text {Spec}(O_{v})$. In other words, any maximal chain in $\text {Spec}(R)$
covers the maximal chain $\text {Spec}(O_{v})$ in a one-to-one correspondence. In particular, for any chain in $\text {Spec}(O_{v})$
there exists a chain in $\text {Spec}(R)$ covering it, in a one-to-one correspondence.

\end{ex}

For more information on quasi-valuations see [Sa1], [Sa2] and [Sa3].

Let $S$ be a ring 
such that $\text {Spec}(S)$ is totally ordered by inclusion, let $R$ be a ring, and let $f:S \rightarrow R$ be a unitary homomorphism with $f[S] \subseteq Z(R)$.
In light of Theorem \ref{D is a maximal chain} and Example \ref{example from qv theory}, if $f$ satisfies GU, GD, and SGB, then $f$ satisfies LO. Compare this example with Lemma
\ref{GU + f^-1[0] contained in P implies LO} and Lemma \ref{GD + f[P]R neq R implies LO}.


\section {Maximal $\mathcal D$-chains}

We shall now study the subject from the opposite point of view: we
take $\mathcal D$, a chain of prime ideals of $S$ and
study $\mathcal D$-chains; in particular, maximal $\mathcal D$-chains.

we start with the definition of a maximal $\mathcal D$-chain.

\begin{defn} Let $\mathcal D$
be a chain of prime ideals of $S$ and let $\mathcal C$ be a $\mathcal D$-chain.
We say that $\mathcal C$ is a {\it maximal $\mathcal D$-chain} (not to be confused with a maximal chain) if
whenever $\mathcal C'$ is a chain of prime ideals of $R$ strictly containing $\mathcal C$ then
there exists $Q \in \mathcal C'$ such that $f^{-1}[Q]\notin \mathcal
D$. Namely, $\mathcal C$ is a $\mathcal D$-chain which is maximal with respect to containment.
\end{defn}

We shall now prove a basic lemma, the existence of maximal $\mathcal D$-chains.

\begin{lem} \label{there exists C maximal over D} $f$ satisfies LO if and only if for every nonempty chain
$\mathcal D \subseteq \text {Spec} (S)$, there exists a
nonempty maximal $\mathcal D$-chain.

\end{lem}

\begin{proof} $(\Rightarrow)$ Let $\mathcal D \subseteq \text {Spec} (S)$ be a nonempty chain, let $P \in
\mathcal D$,
and let $$\mathcal Z= \{ \mathcal E \subseteq \text {Spec} (R) \mid \mathcal E \text{ is a nonempty } \mathcal  D \text{ -chain} \}.$$ By LO there exists $Q \in \text {Spec} (R)$ such that $f^{-1}[Q]=P$;
hence $\{ Q\} \in \mathcal Z$. Therefore, $\mathcal Z \neq
\emptyset$. Now, $\mathcal Z$ with the partial order of
containment satisfies the conditions of Zorn's Lemma and thus
there exists $\mathcal C \in \mathcal Z$ maximal with
respect to containment.

$(\Leftarrow)$ It is obvious.

\end{proof}

Note that one can prove a similar version of $(\Rightarrow)$ of Lemma \ref{there exists C maximal over D} without the LO assumption. However, in this case a maximal $\mathcal D$-chain might be empty. Also note that similarly, one can prove that for any $\mathcal D$-chain $\mathcal C'$ there exists a maximal $\mathcal D$-chain $\mathcal C$ containing $\mathcal C'$.

Our immediate goal is to provide a preliminary connection between the properties LO, INC, GU, GD and SGB and maximal $\mathcal D$-chains. In order to do that, we define the following definition.

\begin{defn} \label{def_layers} Let $n \in \Bbb N$. We say that $f$ satisfies the {\it layer $n$ property} if for every
chain $\mathcal D \subseteq \text {Spec}(S)$ of cardinality $n$, every maximal $\mathcal D$-chain is of cardinality $n$.
\end{defn}

We shall prove that if $f$ satisfies layers $1$, $2$ and $3$ properties then for every
chain $\mathcal D \subseteq \text {Spec}(S)$ of arbitrary cardinality, every maximal $\mathcal D$-chain is of the same cardinality.

The following proposition is seen by an easy inspection. 

\begin{prop} \label{layers iff properties}

1. $f$ satisfies the layer $1$ property if and only if $f$ satisfies LO and INC.

2. $f$ satisfies layers $1$ and $2$ properties if and only if $f$ satisfies LO, INC, GU and GD.

3. $f$ satisfies layers $1$, $2$ and $3$ properties if and only if $f$ satisfies LO, INC, GU, GD and SGB.
\end{prop}

For each of the properties GD, GU and SGB we present now necessary and sufficient conditions in terms of maximal $\mathcal D$-chains.

\begin{prop} \label{mini GD} $f$ satisfies GD iff
for every nonempty chain $\mathcal D \subseteq \text {Spec}(S)$ and for every nonempty maximal $\mathcal D$-chain, $C \subseteq \text {Spec}(R)$,
the following holds: for every $P \in
\mathcal D$ there exists $Q \in \mathcal C$ such that $f^{-1}[Q]
\subseteq P$.


\end{prop}

\begin{proof} $(\Rightarrow)$ Let $\mathcal D \subseteq \text {Spec}(S)$ be a nonempty chain, let 
$C \subseteq \text {Spec}(R)$ be a nonempty maximal $\mathcal D$-chain, and let $P \in \mathcal
D$. Note that $f^{-1}[Q] \in D$ for every $Q \in \mathcal C$ and
since $\mathcal D$ is a chain, for every $Q \in \mathcal C$ we
have $f^{-1}[Q] \subseteq P$ or $P \subseteq f^{-1}[Q]$. Assume to
the contrary that $P \subset f^{-1}[Q]$ for all $Q \in \mathcal C$.
Thus, $$P \subseteq \bigcap_{Q \in \mathcal C} (f^{-1}[Q])=f^{-1}[\bigcap_{Q \in \mathcal C} Q].$$ Obviously, $\bigcap_{Q
\in \mathcal C} Q \subseteq Q$ for every $Q \in \mathcal C$ and
$\bigcap_{Q \in \mathcal C} Q \in \text {Spec}(R)$ (note that
$\mathcal C$ is not empty). Therefore, it is impossible that $P = f^{-1}[\bigcap_{Q \in \mathcal C} Q]$, since then
$\mathcal C \cup \{ \bigcap_{Q \in \mathcal C} Q \}$ is a $\mathcal D$-chain strictly containing $\mathcal C$.

Hence, we may assume that $P \subset f^{-1}[\bigcap_{Q \in \mathcal C} Q]$. Now, $\bigcap_{Q \in \mathcal C} Q$ is a prime ideal of $R$ lying over $f^{-1}[\bigcap_{Q \in \mathcal C} Q] \in \text {Spec}(S)$ (note that $f^{-1}[\bigcap_{Q \in \mathcal C} Q] \subseteq f^{-1}[Q]$ for all $Q \in \mathcal C$ and $\mathcal C$ is a nonempty $\mathcal D$-chain; thus $f^{-1}[\bigcap_{Q \in \mathcal C} Q] \neq S$). Therefore, by GD there exists a prime ideal $Q' \subset \bigcap_{Q \in \mathcal C} Q$ that is lying over $P$. We then get a $\mathcal D$-chain $\mathcal C \cup
\{ Q' \} $ that strictly contains $\mathcal C$, which is again
impossible.

$(\Leftarrow)$ Let $P_{1} \subset P_{2} \in \text {Spec}(S)$ and let $Q_{2} \in \text {Spec}(R)$ lying over $P_{2}$.
Assume to the contrary that there is no prime ideal $Q_{1} \subset Q_{2}$ lying over $P_{1}$. 
Let $\mathcal C$ denote a maximal $\{ P_{1}, P_{2}\}$-chain containing $Q_{2}$.
Hence, for all $Q \in \mathcal C$, $P_{1} \subset f^{-1}[Q](=P_{2})$, a contradiction.

\end{proof}

We shall now prove the dual of Proposition \ref{mini GD}. 

\begin{prop} \label{mini GU} $f$ satisfies GU iff
for every nonempty chain $\mathcal D \subseteq \text {Spec}(S)$ and for every nonempty maximal $\mathcal D$-chain, $C \subseteq \text {Spec}(R)$,
the following holds: for every $P \in
\mathcal D$ there exists $Q \in \mathcal C$ such that $P
\subseteq f^{-1}[Q]$.


\end{prop}

\begin{proof} $(\Rightarrow)$ 
Let $\mathcal D \subseteq \text {Spec}(S)$ be a nonempty chain and let
$C \subseteq \text {Spec}(R)$ be a nonempty maximal $\mathcal D$-chain.
Let $P \in \mathcal
D$ and assume to the contrary that $P \supset f^{-1}[Q] $ for all $Q \in \mathcal
C$. Then,
$$ P \supseteq \bigcup_{Q \in \mathcal C} (f^{-1}[Q])=f^{-1}[\bigcup_{Q
\in \mathcal C} Q].$$ 
Note that $Q \subseteq \bigcup_{Q
\in \mathcal C} Q$ for every $Q \in \mathcal C$ (although $\bigcup_{Q
\in \mathcal C} Q$ is not necessarily a prime ideal of $R$). Now, by Lemma \ref{there exists a minimal prime},
there exists a minimal prime ideal $Q'$ over $\bigcup_{Q
\in \mathcal C} Q$ and by Proposition \ref{minimal prime over union of chain intersection is the same}, $f^{-1}[ \bigcup_{Q \in \mathcal C} Q])=f^{-1}[Q']$, which is a prime ideal of $S$ (note that $f^{-1}[ \bigcup_{Q \in \mathcal C} Q] \subseteq P$, so one cannot have $f^{-1}[ \bigcup_{Q \in \mathcal C} Q])=S$).


Now, if $ P = f^{-1}[ \bigcup_{Q \in \mathcal C} Q])$ then $\mathcal C \cup \{ Q' \}$ is a $\mathcal D$-chain strictly containing $\mathcal C$, a contradiction. If $f^{-1}[ \bigcup_{Q \in \mathcal C} Q]) \subset P$ then by GU, there
exists a prime ideal of $R$, $Q' \subset Q''$ lying over $P$. So,
we get a $\mathcal D$-chain, $\mathcal C \cup \{Q''\}$ strictly
containing $\mathcal C$, which is again impossible.

$(\Leftarrow)$ The proof is almost identical to the proof of $(\Leftarrow)$ in Proposition \ref{mini GD}.

\end{proof}


\begin{prop} \label{mini SGB} $f$ satisfies SGB iff
for every nonempty chain $\mathcal D \subseteq \text {Spec}(S)$, for every nonempty maximal $\mathcal D$-chain, $C \subseteq \text {Spec}(R)$,
and for every cut $\mathcal A$ of
$\mathcal C$ such that $\mathcal A \neq (\emptyset, \mathcal C)$ and $\mathcal A \neq
(\mathcal C,\emptyset )$ (thus $|\mathcal C| \geq 2$), the following holds:
for every $P \in
\mathcal D$ there exists $Q_{1} \in \mathcal A^{L}$ such that $P
\subseteq f^{-1}[Q_{1}]$ or there exists $Q_{2} \in \mathcal A^{R}$ such
that $f^{-1}[Q_{2}] \subseteq P$.


\end{prop}

\begin{proof} $(\Rightarrow)$ Let $\mathcal D \subseteq \text {Spec}(S)$ be a nonempty chain, let
$C \subseteq \text {Spec}(R)$ be a nonempty maximal $\mathcal D$-chain,
and let $\mathcal A$ be a cut of
$\mathcal C$ such that $\mathcal A \neq (\emptyset, \mathcal C)$ and $\mathcal A \neq
(\mathcal C,\emptyset )$.
Assume to the contrary that there exists $P \in \mathcal
D$ such that $$f^{-1}[ Q_{l}] \subset P \subset f^{-1}[ Q_{r}]$$ for
all $Q_{l} \in \mathcal A^{L}$ and $Q_{r} \in \mathcal A^{R}$. Thus,
$$f^{-1}[\bigcup_{Q \in \mathcal A^{L}} Q]=\bigcup_{Q \in \mathcal A^{L}} f^{-1}[Q] \subseteq P \subseteq
\bigcap_{Q \in \mathcal A^{R}} f^{-1}[Q]=f^{-1}[\bigcap_{Q \in \mathcal A^{R}} Q].$$
Now, $\bigcap_{Q \in \mathcal A^{R}} Q \in \text {Spec}(R)$ and $\mathcal C \cup \{ \bigcap_{Q \in \mathcal A^{R}} Q \}$
is a chain of prime ideals of $R$. Thus, since $\mathcal C$ is a maximal $\mathcal D$-chain, $\bigcap_{Q \in \mathcal A^{R}} Q$ is not lying over $P$. Hence,
$$P \subset f^{-1}[\bigcap_{Q \in \mathcal A^{R}} Q)].$$ 

On the other hand, $\bigcup_{Q \in \mathcal A^{L}} Q$ is an ideal of $R$ lying over $f^{-1}[ \bigcup_{Q \in \mathcal A^{L}} Q] $ 
 and is contained in the prime ideal $\bigcap_{Q \in \mathcal A^{R}} Q$.
Thus, by Lemma \ref{there exists a minimal prime}, there exists a minimal prime ideal $\bigcup_{Q \in \mathcal A^{L}} Q \subseteq Q' \subseteq \bigcap_{Q \in \mathcal A^{R}} Q$ over $\bigcup_{Q \in \mathcal A^{L}} Q$, and by Proposition \ref{minimal prime over union of chain intersection is the same}, $Q'$ is lying over $f^{-1}[ \bigcup_{Q \in \mathcal A^{L}} Q] \in \text {Spec}(S)$ (note that, as in Proposition \ref{mini GU}, $f^{-1}[ \bigcup_{Q \in \mathcal A^{L}} Q] \subseteq P$, so one cannot have $f^{-1}[ \bigcup_{Q \in \mathcal C} Q])=S$). Now, since $\mathcal C \cup \{ Q' \}$ is a chain of prime ideals and $\mathcal C$ is a maximal $\mathcal D$-chain,
one cannot have $Q'$ lying over $P$. Thus, $$ f^{-1}[\bigcup_{Q \in \mathcal A^{L}} Q]  \subset P.$$

So we have a chain of prime ideals $ f^{-1}[\bigcup_{Q \in \mathcal A^{L}} Q]  \subset P  \subset f^{-1}[\bigcap_{Q \in \mathcal A^{R}} Q)]$ in $\text {Spec}(S)$
(note that, as in Proposition \ref{mini GD}, $f^{-1}[\bigcap_{Q \in \mathcal A^{R}} Q)] \neq S$); we have
$Q' \subset \bigcap_{Q \in \mathcal A^{R}} Q \in \text {Spec}(R)$ lying over $f^{-1}[\bigcup_{Q \in \mathcal A^{L}} Q]$ and $f^{-1}[\bigcap_{Q \in \mathcal A^{R}} Q]$
respectively, such that $\mathcal C \cup \{ Q' , \bigcap_{Q \in \mathcal A^{R}} Q\} $ is a chain in $\text {Spec}(R)$. Thus,
by SGB, there exists $Q'' \in \text {Spec}(R)$ such that
$$ Q' \subset Q'' \subset \bigcap_{Q \in
\mathcal A^{R}} Q$$ and $f^{-1}[Q'']=P$. But then $\mathcal C \cup \{ Q''\}$
is a $\mathcal D$-chain strictly containing $\mathcal C$, a
contradiction.

$(\Leftarrow)$ Let $P_{1} \subset P_{2} \subset P_{3} \in \text {Spec}(S)$ and let $Q_{1} \subset Q_{3} \in \text {Spec}(R)$ such that $Q_{1}$ is lying over $P_{1}$
and $Q_{3}$ is lying over $P_{3}$.
Assume to the contrary that there is no prime ideal $Q_{1} \subset Q_{2} \subset Q_{3}$ lying over $P_{2}$.
Let $\mathcal C$ denote a maximal $\{ P_{1}, P_{2}, P_{3}\}$-chain containing $\{ Q_{1},Q_{3}\}$.
Let $\mathcal A$ be the cut defined by $\mathcal A^{L} =\{ Q \in \mathcal C \mid f^{-1}[Q] =P_{1}\}$ and
$\mathcal A^{R} =\{ Q \in \mathcal C \mid f^{-1}[Q] =P_{3}  \}$ (note that, by our assumption, for all $Q \in \mathcal C, f^{-1}[Q] \neq P_{2})$.
Note that $\mathcal A^{L} \neq \emptyset$ and $\mathcal A^{R} \neq \emptyset$.
Hence, $ f^{-1}[Q_{l}] \subset P_{2} \subset f^{-1}[Q_{r}]$ for all $Q_{l} \in \mathcal A^{L}$ and $Q_{r} \in \mathcal A^{R}$, a contradiction.

\end{proof}

Using Propositions \ref{mini GD} and \ref{mini SGB} the following corollary is obvious.

\begin{cor} Assume that $f$ satisfies GD and SGB; then $f$ satisfies GGD (generalized going down).
Moreover, if $\mathcal D \subseteq \text {Spec}(S)$ is a chain of prime ideals whose final element is $P$
and $Q \in \text {Spec}(R)$ lying over $P$, then every maximal $\mathcal D$-chain
containing $Q$ (and in particular if $Q$ is its final element) is a cover of $\mathcal D$.

\end{cor}

Note that the dual of the previous corollary is also valid.

We shall now present sufficient conditions for a GU (and GD) homomorphism to satisfy LO.

\begin{lem} \label{GU + f^-1[0] contained in P implies LO} Assume that $f$ satisfies GU and let $P \in \text {Spec}(S)$ such that
$f^{-1}[\{ 0 \} ] \subseteq P$. Then there exists $Q \in \text {Spec}(R)$ lying over $P$; in particular, if $f^{-1}[\{ 0 \} ] \subseteq P$ for all $P \in \text {Spec}(S)$,
then $f$ satisfies LO.
\end{lem}

\begin{proof} The set $\mathcal Z=\{ I \lhd R \mid f^{-1} [ I] \subseteq P\}$ is not empty and satisfies the conditions of Zorn's Lemma.
Thus, there exists a maximal element $Q' \in \mathcal Z $; it is standard to check that $Q'$ is a prime ideal of $R$.
Now, $Q'$ must be lying over $P$; since otherwise, by GU there exists $Q' \subset Q \in \text {Spec}(R)$ lying over $P$.
 \end{proof}

\begin{lem} \label{GD + f[P]R neq R implies LO} Assume that $f$ is unitary and satisfies GD.
Let $P \in \text {Spec}(S)$ such that
$f[P]R \neq R$. Then there exists $Q \in \text {Spec}(R)$ lying over $P$; in particular, $f^{-1}[f[P]R]=P$. Moreover, if $f[P]R \neq R$ for all $P \in \text {Spec}(S)$,
then $f$ satisfies LO.
\end{lem}

\begin{proof} Let $Q' \in \text {Spec}(R)$ be any prime ideal containing $f[P]R$. It is obvious that $P \subseteq f^{-1}[f[P]R]$;
so, $Q'$ is lying over a prime ideal $P'$ (note that $f$ is unitary) such that $P \subseteq P'$. By GD, there exists
a prime ideal $Q \subseteq Q'$ lying over $P$. Clearly, $f[P]R \subseteq Q$; hence, $f^{-1}[f[P]R]=P$. The last assertion is clear.  \end{proof}

\begin{defn} Let $\mathcal D \subseteq \text {Spec}(S)$ be a chain of prime ideals
and let $\mathcal C  \subseteq \text {Spec}(R)$ be a $\mathcal D$-chain. We say that $\mathcal C$ is a {\it maximal cover} of
$\mathcal D$ if $\mathcal C$ is a cover of
$\mathcal D$, which is maximal with respect to containment.

\end{defn}

\begin{thm} \label{maximal D-chain covers D} $f$
satisfies GD, GU and SGB if and only if for every nonempty chain $\mathcal D \subseteq \text {Spec}(S)$ and for every
nonempty maximal $\mathcal D$-chain $\mathcal C  \subseteq \text {Spec}(R)$, $\mathcal C$ is a (maximal) cover of
$\mathcal D$.

\end{thm}

\begin{proof} $(\Rightarrow)$ Let $\mathcal D \subseteq \text {Spec}(S)$ be a nonempty chain and
let $\mathcal C  \subseteq \text {Spec}(R)$ be a nonempty maximal $\mathcal D$-chain. 
Assume to the contrary that $\mathcal C$ is
not a cover of $\mathcal D$. So there
exists a prime ideal $P \in \mathcal D \setminus \{f^{-1}[Q]\}_{Q \in \mathcal C}$.
Obviously $\{f^{-1}[Q]\}_{Q \in \mathcal C} \cup \{ P\}$ is a chain. We have the following three possibilities:

1. $P \subset f^{-1}[Q]$ for all $Q \in \mathcal C$. However, by assumption $f$ satisfies GD and thus by Proposition \ref{mini GD},
this situation is impossible.

2. $f^{-1}[Q] \subset P$ for all $Q \in \mathcal C$. However, by assumption $f$ satisfies GU and thus by Proposition \ref{mini GU},
this situation is impossible.

3. There exist $Q', Q'' \in \mathcal C$ such that $f^{-1}[Q'] \subset P
\subset f^{-1}[Q'']$. So, let $\mathcal A$ be the cut defined by $\mathcal A^{L} =\{ Q \in \mathcal C \mid f^{-1}[Q] \subset P \}$ and
$\mathcal A^{R} =\{ Q \in \mathcal C \mid P \subset f^{-1}[Q]  \}$. Note that $\mathcal A^{L} \neq \emptyset$ and $\mathcal A^{R} \neq \emptyset$.
Thus $f^{-1}[ Q_{l}] \subset P \subset f^{-1}[ Q_{r}]$ for
all $Q_{l} \in \mathcal A^{L}$ and $Q_{r} \in \mathcal A^{R}$. However, by assumption $f$ satisfies SGB and thus by Proposition \ref{mini SGB},
this situation is impossible.

Finally, it is obvious that $\mathcal C$ is a maximal cover of $\mathcal D$, since $\mathcal C$ is a cover of $\mathcal D$ and a maximal $\mathcal D$-chain.

$(\Leftarrow)$ It is obvious.



\end{proof}

We note that by [KO, Proposition 4 and Corollary 11], if $S \subseteq R$ are commutative rings such that $R$ satisfies GU over $S$, then for
every chain of prime ideals $\mathcal D \subseteq \text {Spec}(S)$ there exists
a chain of prime ideals in $\text {Spec}(R)$ covering it. However, if for example $R$ does not satisfy GD over $S$, then one can find 
a chain $\mathcal D \subseteq \text {Spec}(S)$ such that not \textbf{every} maximal $\mathcal D$-chain is a cover of $\mathcal D$.
See Example \ref{ex GU LO but no GD}.

We shall now present one of the main results of this paper.

\begin{thm} \label{LO, INC, GD, GU and SGB imply perfect cover} Assume that $f$
satisfies LO, INC, GD, GU and SGB. Let $\mathcal D$ denote a chain of prime ideals in $\text {Spec} (S)$. Then there
exists a perfect maximal cover of $\mathcal D$. Moreover, any maximal $\mathcal
D$-chain is a perfect maximal cover.

\end{thm}

\begin{proof} If $\mathcal D$ is empty then a maximal $\mathcal D$-chain must be empty and the assertion is clear.
We may therefore assume that $\mathcal D$ is not empty. $f$ satisfies LO; thus by Lemma \ref{there exists C maximal over D} there exists a nonempty maximal $\mathcal D$-chain.
Now, let $\mathcal C$ be any maximal $\mathcal D$-chain.
$f$ satisfies GD, GU and SGB and therefore by Theorem \ref{maximal D-chain covers D}, $\mathcal C$ is a maximal cover of $\mathcal D$.
Finally,  $f$ satisfies INC; hence by Lemma \ref{mini inc}, $\mathcal C$ is a perfect maximal cover of $\mathcal D$.

\end{proof}

We can now present another main result of this paper. The following corollary closes the circle.

\begin{cor} \label{equivalent conditions} The following conditions are equivalent:

1. $f$ satisfies layers 1,2 and 3 properties.

2. $f$ satisfies LO, INC, GD, GU and SGB.

3. For every chain $\mathcal D \subseteq \text {Spec} (S)$ and for every maximal $\mathcal
D$-chain $\mathcal C \subseteq \text {Spec} (R)$, $\mathcal C$ is a perfect maximal cover of $\mathcal D$.

4. For every chain $\mathcal D \subseteq \text {Spec} (S)$ and for every maximal $\mathcal
D$-chain $\mathcal C \subseteq \text {Spec} (R)$, $|C|=|D|$.

\end{cor}

\begin{proof} $(1) \Rightarrow (2)$: By Proposition \ref{layers iff properties}. $(2) \Rightarrow (3)$: By Theorem \ref{LO, INC, GD, GU and SGB imply perfect cover}.
$(3) \Rightarrow (4)$ and $(4) \Rightarrow (1)$ are obvious.

\end{proof}


\begin{ex} Notation and assumptions as in Example \ref{example from qv theory}. Then $R$ satisfies
LO, INC, GU, GD, and SGB over $O_{v}$. Therefore, for every chain $\mathcal D \subseteq \text {Spec} (O_{v})$ and for every maximal $\mathcal
D$-chain $\mathcal C \subseteq \text {Spec} (R)$, $\mathcal C$ is a perfect maximal cover of $\mathcal D$.
\end{ex}


We note that one can easily construct examples of functions that satisfy some of the properties. We shall present now an example in
which the cardinality of the set of all chains in $ \text {Spec} (S)$ that satisfy property 4 of Corollary \ref{equivalent conditions} is equal to
the cardinality of $ P(\text {Spec} (S))$, the power set of $ \text {Spec} (S)$, still without $f$ satisfying the equivalent conditions of Corollary \ref{equivalent conditions} (because property 4 is not fully satisfied).

\begin{ex} Let $A$ be a valuation ring of a field $F$ with $ |\text {Spec} (A)|=a$, an infinite cardinal. Let $Q_{0}$ denote the maximal ideal of $A$ and assume that
$Q_{0}$ has an immediate predecessor, namely, a prime ideal $Q_{1} \subset Q_{0}$, such that there is no prime ideal between $ Q_{0}$ and $Q_{1}$.
Let $B=A_{Q_{1}}$, a valuation ring of $F$ with maximal ideal $Q_{1}$. Consider $A$ as a subring of $B$ and let $f : A \rightarrow B$ be the map defined by $f(a)=a$ for all $a \in A$. It is well known that there is a bijective map $Q \rightarrow Q \cap A$ from $\text {Spec} (B)$ to $\text {Spec} (A) \setminus \{ Q_{0}\}$.
Let $$T=\{ \mathcal D \subseteq \text {Spec} (A) \mid  \text { for every maximal } \mathcal
D  \text {-chain } \mathcal C \subseteq \text {Spec} (B), \ |C|=|D| \};$$
It is not difficult to see that
$$P(\text {Spec} (A)) \setminus T= \{  \mathcal E \subseteq \text {Spec} (A) \mid Q_{0} \in \mathcal E, |\mathcal E|=n \text { for some } n \in \Bbb N \}.$$
It is obvious that $|P(\text {Spec} (A)) \setminus T|=a$ and thus
$|T|=|P(\text {Spec} (A))|=2^{a}$. 
Note that $f$ does not satisfy LO and GU, although it does satisfy INC, GD and SGB.

Finally, denote $T'=\{ \mathcal D \subseteq \text {Spec} (A) \mid  \text { for every maximal } \mathcal
D  \text {-chain } \mathcal C \subseteq \text {Spec} (B), \ \mathcal C \text { is a perfect maximal cover of } \mathcal D \}$. Then
$T' \subset T$; in fact, $T'=P(\text {Spec} (A) \setminus \{ Q_{0}\} ) $. In other words, the set of chains in $\text {Spec} (A)$ that satisfy
property 3 of Corollary \ref{equivalent conditions} is strictly contained in the set of chains in $\text {Spec} (A)$ that satisfy
property 4 of Corollary \ref{equivalent conditions}.

\end{ex}


We close this paper by presenting some results regarding maximal chains in $\text {Spec}(S)$.
We shall need to assume that $f$ is unitary.

\begin{lem} \label{maximal cover over maximal is maximal} Assume that $f$ is unitary. Let $\mathcal D \subseteq \text {Spec}(S)$ be a maximal chain
and let $\mathcal C  \subseteq \text {Spec}(R)$ be a maximal cover of $\mathcal D$. Then $\mathcal C$ is a maximal chain in $\text {Spec}(R)$.

\end{lem}

\begin{proof} Assume to the contrary that $\mathcal C$ is
not a maximal chain in $\text {Spec}(R)$. Then there exists a prime ideal $Q' $
of $R$ such that $\mathcal C \cup \{ Q' \}$ is a chain of prime
ideals of $R$ strictly containing $\mathcal C$. However, since $\mathcal C$ is a cover of $\mathcal D$, we have
$\mathcal D = \{ f^{-1}[Q]\}_{Q \in \mathcal C}$ and since $\mathcal C$ is a maximal cover of $\mathcal D$, we have
$f^{-1}[Q'] \notin \mathcal D$. Note that, since $f$ is unitary, $f^{-1}[Q'] \in \text {Spec}(S)$. Therefore, $$  \{ f^{-1}[Q]\}_{Q \in \mathcal C} \cup \{ f^{-1}[Q']\}$$ is a chain in
$\text {Spec}(S)$ strictly containing $\mathcal D$, a contradiction.

\end{proof}

We are now able to prove that, assuming $f$ is unitary and
satisfies GD, GU and SGB, if $\mathcal
D$ is a maximal chain in $\text {Spec}(S)$ and $\mathcal C$ is any maximal $\mathcal D$-chain, then
$\mathcal C$ must be a maximal chain in $\text {Spec}(R)$.

\begin{cor} \label{maximal D-chain implies maximal} Assume that $f$ is unitary and
satisfies GD, GU and SGB. Let $\mathcal
D$ be a maximal chain in $\text {Spec}(S)$ (so $\mathcal
D$ is not empty) and let $\mathcal C \subseteq \text {Spec}(R)$
be a nonempty maximal $\mathcal D$-chain. Then $\mathcal C$ is
a maximal cover of $\mathcal
D$ and a maximal chain in $\text {Spec}(R)$. 

\end{cor}

\begin{proof} By Theorem \ref{maximal D-chain covers D}, $\mathcal C$ is a maximal cover of $\mathcal D$. By Lemma \ref{maximal cover over maximal is maximal}, $\mathcal C$ is a maximal chain in $\text {Spec}(R)$.

\end{proof}

\begin{cor} \label{there exists maximal D-chain which is maximal} Assume that $f$ is unitary and
satisfies LO, GD, GU and SGB. Let $\mathcal D$ denote a maximal chain of prime ideals in $\text {Spec}(S)$. Then there
exists a maximal cover of $\mathcal D$, which is a maximal chain of
prime ideals in $\text {Spec}(R)$.

\end{cor}

\begin{proof} By Lemma \ref{there exists C maximal over D} there exists a nonempty maximal $\mathcal
D$-chain, and by Corollary \ref{maximal D-chain implies maximal} it
is a maximal chain of prime ideals in $\text {Spec}(R)$, covering $\mathcal D$.

\end{proof}

Note: it is not difficult to see that if $f$ is not unitary then the results presented in \ref{maximal cover over maximal is maximal},
\ref{maximal D-chain implies maximal} and \ref{there exists maximal D-chain which is maximal} are no longer valid.

Department of Mathematics, Sce College, Ashdod 77245, Israel.

{\it E-mail address: sarusss1@gmail.com}

\end{document}